\documentclass[11pt, draft]{amsart}
\usepackage{amssymb, amstext, amscd, amsmath, amssymb}
\usepackage{mathtools, color, paralist, dsfont, rotating}
\usepackage{verbatim}
\usepackage{enumerate}
\usepackage{tikz-cd}
\usepackage{tikz}
\usepackage[all]{xy}
\numberwithin{equation}{section}

\makeatletter
\def\@cite#1#2{{\m@th\upshape\bfseries%
[{#1\if@tempswa{\m@th\upshape\mdseries, #2}\fi}]}}
\makeatother
\theoremstyle{plain}
\newtheorem{theorem}{Theorem}[section]
\newtheorem{corollary}[theorem]{Corollary}
\newtheorem{proposition}[theorem]{Proposition}
\newtheorem{lemma}[theorem]{Lemma}
\theoremstyle{definition}
\newtheorem{definition}[theorem]{Definition}

\newtheorem{remark}[theorem]{Remark}

\theoremstyle{remark}

\mathtoolsset{centercolon}
  \newcommand{\cA}{{\mathcal{A}}}
  \newcommand{\cB}{{\mathcal{B}}}
  \newcommand{\cC}{{\mathcal{C}}}

  \newcommand{\cF}{{\mathcal{F}}}
  
	\newcommand{\cH}{{\mathcal{H}}}

  \newcommand{\cK}{{\mathcal{K}}}
	\newcommand{\cL}{{\mathcal{L}}}
  
  \newcommand{\cN}{{\mathcal{N}}}
	\newcommand{\cO}{{\mathcal{O}}}

  \newcommand{\cT}{{\mathcal{T}}}

\newcommand{\A}{{\mathcal{A}}}
\newcommand{\B}{{\mathcal{B}}}
\newcommand{\C}{{\mathcal{C}}}

\newcommand{\E}{{\mathcal{E}}}
\newcommand{\F}{{\mathcal{F}}}
\newcommand{\G}{{\mathcal{G}}}
\renewcommand{\H}{{\mathcal{H}}}

\newcommand{\K}{{\mathcal{K}}}
\renewcommand{\L}{{\mathcal{L}}}

\newcommand{\N}{{\mathcal{N}}}
\renewcommand{\O}{{\mathcal{O}}}

\renewcommand{\S}{{\mathcal{S}}}
\newcommand{\T}{{\mathcal{T}}}
\newcommand{\U}{{\mathcal{U}}}
\newcommand{\V}{{\mathcal{V}}}

\newcommand{\bZ}{\mathbb{Z}}

\newcommand{\bbC}{{\mathbb{C}}}

\newcommand{\bbI}{{\mathbb{I}}}

\newcommand{\bbN}{{\mathbb{N}}}

\newcommand{\bbT}{{\mathbb{T}}}
\newcommand{\bbZ}{{\mathbb{Z}}}

\newcommand{\lip}{\langle}
\newcommand{\rip}{\rangle}

\newcommand{\cenv}{\mathrm{C}_e^*}

\newcommand{\Aut}{\operatorname{Aut}}
\newcommand{\alg}{\operatorname{alg}}

\newcommand{\spn}{\operatorname{span}}

\newcommand{\supp}{\operatorname{supp}}

\newcommand\cpr{\rtimes_{\alpha}^{r}\,{\mathcal{G}}}
\newcommand\cpf{\rtimes_{\alpha}\, {\mathcal{G}}}

\newcommand{\ca}{\mathrm{C}^*}

\newcommand{\hf}{\widehat{f}}
\newcommand{\hg}{\widehat{g}}

\begin{document}

\title{Product systems of $\ca$-correspondences and Takai duality}


\author[E. Katsoulis]{Elias Katsoulis}
\address{Department of Mathematics \\ East Carolina University \\ Greenville \\ NC \\ 27858-4353 \\ USA}
\email{katsoulise@ecu.edu}



\maketitle

\begin{abstract}
We establish the Hao-Ng isomorphism for generalized gauge actions of locally compact abelian groups on product systems over abel-\break ian lattice orders and we then use it to explore Takai duality in this context. As an application we generalize some recent work of Schafhauser.
\end{abstract}

\section{Introduction}
If $(\A, G, \alpha)$ is a $\ca$-dynamical system over a locally compact abelian group $G$, then the crossed product $\ca$-algebra $\A\rtimes_{\alpha} G$ admits a natural action $\widehat{\alpha}$ of the dual group $\widehat{G}$. This produces a new $\ca$-dynamical system $(\A\rtimes_{\alpha} G, \widehat{G}, \widehat{\alpha})$ and the classical Takesaki-Takai duality asserts that
\[
(\A \rtimes_{\alpha} G)\rtimes_{\widehat{\alpha}}\widehat{G} \simeq \A\otimes \K(L^2(G, \mu)),
\]
where $\K(L^2(G, \mu))$ denotes the compact operators on $L^2(G , \mu)$, $\mu$ the Haar measure on $G$.

The Cuntz-Pimsner algebras of product systems of $\ca$-correspondences over abelian lattice orders $(G,P)$, include as particular examples crossed products of $\ca$-algebras by various discrete abelian  groups. Just as in the case of a crossed product, the generic Cuntz-Pimsner algebra $\N\O_X^r$ admits a natural action of the dual group $\widehat{G}$, the so called gauge action. One would like to calculate the associated crossed product  $\N\O_X^r \rtimes \widehat{G}$; in particular one wonders what is the analogue of the classic Takesaki-Takai duality in this case. This is the main theme of this paper with inspiration coming from the work of Abadie~\cite{Ab} who first studied this problem for Cuntz-Pimsner algebras of $\ca$-correspondences.

A key step in the calculation of $\N\O_X^r \rtimes \widehat{G}$ involves the solution of the Hao-Ng isomorphism problem for generalized gauge actions of locally compact abelian groups on product systems over abelian lattice orders. (See Section~\ref{s;HaoNg} for a statement of the Hao-Ng isomorphism problem and \cite{KR} for a detailed discussion of its impact on current operator algebra research.) Recent advances in the theory of non-selfadjoint operator algebras by Dor-On and the author \cite{DorK} allow us to do this in Theorem~\ref{HaoNgreduced}. A key element in the proof of Theorem~\ref{HaoNgreduced} is Proposition~\ref{p;criterion} which gives a very workable criterion for checking the compact alignment of a product system.

Using the Hao-Ng isomorphism and the Fourier transform, we calculate the crossed product $\N\O_X^r \rtimes \widehat{G}$ by the gauge action. Indeed in Theorem~\ref{thm;calc} we give a very concrete picture for $\N\O^r_{X}\rtimes \widehat{G}$ as the Cuntz-Pimsner algebra of a product system that involves only the group $G$  and not its dual. For product systems of regular and full $\ca$-correspondences over abelian lattice orders we can do more. In Corollary~\ref{cor;TakaiDuality} we show that $\N\O_X^r \rtimes \widehat{G}$ is Morita equivalent to the core of $\N\O_X^r$, thus offering a true generalization of the Takai duality in our context. As a consequence of Corollary~\ref{cor;TakaiDuality}, we are able to generalize some recent results of Schafhauser \cite{Sch} in our context; this is discussed in Section~\ref{s;Sch}.

\section{preliminaries}

The purpose of this section is to establish notation and provide the basic definitions and results  that are necessary for the rest of the paper. 

The study of $\ca$-correspondences and associated $\ca$-algebras began with the work of Pimsner 
\cite{Pim97} and was systematized by Katsura \cite{Kats02, Kat04a, Kat07}, following work of Muhly and Solel~\cite{MS98b} and others. Inspired by work of Nica ~\cite{Nic92}, a more general class of $\ca$-algebras, those associated with product systems, has also been of interest and is under investigation by many specialists \cite{CLSV11, Fow99, Fow02, FR98}. All these abstract operator algebras were inspired by concrete manifestations, which are at the forefront of the theory and continue to offer insight for the general theory \cite{KumP, KPR, KumP2, LRaeb, RS03, RSY03, RSY04, Yee06}.
The literature on product systems is vast and providing a comprehensive summary of the theory in this short paper was proven to be an impossible task. The author apologizes for the inevitable omission of some central results in the theory from the exposition below. 

In what follows, if $M$ is a subset of a normed vector space $\V$, then $[M]$ will denote the closed linear subspace of $\V$ generated by $M$. If $\V$ also happens to be an algebra and $M, N\subseteq  \V$, then $MN$ will denote the closed subspace of $\V$ generated by products of the form $mn$, $m \in  M, n \in N$. An ideal of a $\ca$-algebra always means a closed two-sided ideal.

\subsection{C*-correspondences}
 Here we will give an overview of Hilbert C*-correspondences. For further details and material, we recommend \cite{Lan95}.
 
Let $\cA$ be a C*-algebra. A right inner product $\cA$-module is a complex vector space $X$ equipped with a right action of $\cA$ and an $\cA$-valued map $\lip \cdot , \cdot \rip : X\times X \rightarrow \cA$  which is $\cA$-linear in the second argument, such that for $x,y \in X$ we have 
\begin{enumerate}
\item
$\lip x,x\rip \geq 0$
\item
$\lip x,x \rip = 0$ if and only if $x=0$
\item
$\lip x,y \rip = \lip y,x \rip^*$
\end{enumerate}
When $X$ is complete with respect to the norm given by $\| x \| = \| \lip x,x \rip \|^{\frac{1}{2}}$ we say that $X$ is a Hilbert $\cA$-module. We say that $X$ is full provided that $[\langle X,X\rangle] =\A$, i.e., the closed linear space generated by elements of the form $\langle x, y\rangle$, $x, y \in X$, equals $\A$.

Let $X$ be a Hilbert $\cA$ module. We say that a map $T : X \rightarrow X$ is adjointable if there's a map $T^* :X \rightarrow X$ such that $\lip Tx,y \rip = \lip x,  T^*y \rip$ for every $x,y \in X$. Every adjointable operator is automatically $\cA$-linear and continuous. We denote by $\cL(X)$ the C*-algebra of adjointable operators equipped with the operator norm. For $x,y\in X$ there is a special adjointable operator $\theta_{x,y} \in \cL(X)$ given by $\theta_{x,y}(z) = x  \lip y,z \rip$. We will denote by $\cK(X) \lhd \cL(X)$ the closed ideal of generalized compact operators generated by $\theta_{x,y}$ with $x,y \in X$.

A $\cB$-$\cA$ $\ca$-correspondence is then just a (right) Hilbert $\cA$-module $X$ along with a $*$-homomorphism $\phi : \cB \rightarrow \cL(X)$ which is non-degenerate, i.e., $[\phi(B)X]= X$ (this is sometimes called essential). If $X$ is an $\cA$-$\cA$ $\ca$-correspondence we will just call $X$ an $\cA$-correspondence. We think of $\phi$ as implementing a left action of $\cB$ on $X$, and we will be writing $bx$ for $\phi(b)x$. 

A $\cB$-$\cA$ $\ca$-correspondence is called a Hilbert $\cB$-$\cA$ bimodule if there exists a left $\cB$-valued inner product $[ \cdot , \cdot ] : X\times X \rightarrow \cB$  which is $\cB$-linear in the first argument, such that for $x,y , z \in X$ we have 
\begin{enumerate}
\item
$[ x,x] \geq 0$
\item
$[ x,y ] = \lip y,x \rip^*$
\item
$[x,y]z=x\langle y, z\rangle.$
\end{enumerate}
If $X$ is both left and right full then it is said to be an equivalence bimodule. In that case the action of $\B$ is automatically injective.

When $X$ is a $\cC$-$\cB$-correspondences and $Y$ a $\cB$-$\cA$-correspondence, we may form the interior tensor product $X \otimes_{\cB} Y$. Indeed, let $X\odot_{\cB} Y$ be the algebraic $\cB$-balanced tensor product. Then the formula
$$
\lip x \odot y, w \odot z \rip := \lip y, \lip x,w \rip \cdot z \rip,
$$
determines an $\cA$-valued sesquilinear form on $X\odot_{\cB} Y$, whose Hausdorff completion $X \otimes_{\cB} Y$ is a (right) Hilbert $\cA$-module. There is then a left $\cC$ action $\cC \rightarrow \cL(X \otimes_{\cB} Y)$ given by $c \cdot (x\odot y) = (c \cdot x) \odot y$ for $x \in X$, $y\in Y$ and $c\in \cC$.

\subsection{Product systems over semigroups}

Let $\cA$ be a C*-algebra and $P$ a semigroup with identity $e$. A product system over $P$ with coefficients in $\cA$ is a semigroup of $\cA$-correspondences $X = (X_p)_{p\in P}$ such that
\begin{enumerate}
\item
$X_e = \cA$ is the trivial $\A$-corresponndence.
\item
For $p,q \in P$, there exists a unitary $\A$-linear isomorphism $U_{p,q} : X_p \otimes X_q \rightarrow X_{pq}$
\item
The left and right multiplication on each $X_p$ are given via $U_{e,p}$ and $U_{p,e}$ for each $p\in P$ and we also have associativity in the sense that for $p,q,r \in P$,
$$
U_{p,qr}(I_{X_p} \otimes U_{q,r}) = U_{pq,r}(U_{p,q}\otimes I_{X_r})
$$ 
\end{enumerate}

We will denote $U_{p,q}(x\otimes y) = xy \in X_{pq}$ for every $x\in X_p$ and $y\in X_q$. We will also denote by $\phi_p : X_e \rightarrow \cL(X_p)$ the left action on $X_p$ for each $p\in P$. In particular, $\phi_{pq}(a)(xy) = (\phi_p(a)x)y$ for all $p,q\in P$, $a\in A$ and $x\in X_p$, $y\in X_q$. 

Given $p \in P \setminus \{e\}$ and $q \in P$, the unitary $X_e$-linear map $U_{p,q} : X_p \otimes X_q \rightarrow X_{pq}$ induces a $*$-homomorphism $\iota_p^{pq} : \cL(X_p) \rightarrow \cL(X_{pq})$ via 
$$
\iota_{p}^{pq}:= U_{p,q} \circ (S \otimes id_{X_p})U_{p,q}^{-1}
$$
for each $S\in \cL(X_p)$. Alternatively, we have that the $*$-homomorphism $\iota_p^{pq}$ is given by the formula $\iota_p^{pq}(S)(xy) = (Sx)y$ for each $S\in \cL(X_p)$, $x\in X_p$ and $y\in X_q$. (For notational ease we will be simply writing $S\otimes I$ instead of $\iota_p^{pq}(S)$ whenever $p$ and $q$ are easily understood from the context.) For $\iota_e^q$, we first define on $X_e \cong \cK(X_e)$ via $\iota_e^p(a) = \phi_p(a)$, and then extend uniquely to $\cL(X_e)$ via \cite[Proposition 2.5]{Lan95} to obtain a map $\iota_e^q : \cL(X_e) \rightarrow \cL(X_q)$. 

When $X=(X_p)_{p\in P}$ is a product system over a quasi-lattice ordered semigroup $(G,P)$, we will say that $X$ is \emph{compactly aligned} if whenever $S\in \cK(X_p)$ and $T\in \cK(X_q)$ for some $p,q\in P$ with $p \vee q < \infty$, then $$(S\otimes I)(T\otimes I)= \iota_p^{p \vee q}(S) \iota_q^{p \vee q}(T) \in \cK(X_{p\vee q}).$$

\subsection{Nica-Toeplitz representations}  \label{ss;NT}
We next define representations of compactly aligned product systems over quasi-lattice ordered groups.

\begin{definition} \label{D:isom}
Suppose $(G,P)$ is a quasi-lattice ordered group, and $X= \{X_p\}_{p \in P}$ a compactly aligned product system over $P$. An \emph{isometric representation} of $X$ into a $\ca$-algebra $\B$ is a map $\psi: X \rightarrow \B$ comprised of linear maps $\psi_p : X_p \rightarrow \B$ for each $p\in P$ such that
\begin{enumerate}
\item
$\psi_e$ is a $*$-homomorphism from $X_e$ into $\B$.
\item
$\psi_p(x)\psi_q(y) = \psi_{pq}(xy)$ for all $p,q\in P$ and $x \in X_p$, $y\in X_q$.
\item
$\psi_p(x)^*\psi_p(y) = \psi_e(\lip x,y \rip)$ for all $p\in P$ and $x,y\in X_p$.
\end{enumerate}
\end{definition}

It is standard to show that each $\psi_p$ is contractive, and is isometric precisely when $\psi_e$ is injective. We will say that $\psi$ is non-degenerate provided that $\B \subseteq B(\H)$ and $\psi_e$ is non-degenerate. For each $p \in P$ we have a representation of $\K(X_p)$ that extends the association $\theta_{x, y}\mapsto \psi_p(x)\psi_p(y)^*$, $x , y \in X_p$; by abusing notation we will also denote that representation as $\psi_p$. By by \cite[Proposition 2.5]{Lan95} the representation $\psi_p$ admits a strict-sot continuous extension $\rho_p\colon \L(X_p)\rightarrow B(\H)$.

If $X$ is the trivial product system over $(G,P)$, i.e., $X_p=\bbC$, for all $p \in P$, then an isometric representation of $X$ is simply a representation of $P$ as a semigroup of isometries.

We will say that an isometric representation $\psi: X \rightarrow \B$ of a compactly aligned product system $X$ is \emph{Nica-covariant} if for any $p,q\in P$ and $S\in \cK(X_p)$, $T\in \cK(X_q)$ we have that
\begin{equation} \label{D:NC1}
\psi_p(S)\psi_q(T) = 
\begin{cases}
\psi_{p\vee q}(S\otimes I)((T\otimes I)) & \text{if } p \vee q < \infty  \\
0 & \text{otherwise.}
\end{cases}
\end{equation}

In the case where $\B \subseteq B(\H)$ and $\psi_e$ is a non-degenerate representation the above definition simplifies. Indeed in that case $\psi: X \rightarrow B(\H)$ is Nica-covariant if and only if for any $p,q\in P$ we have 
\begin{equation} \label{D:NC2}
\rho_p(I)\rho_q(I) = 
\begin{cases}
\rho_{p\vee q}(I) & \text{if } p \vee q < \infty  \\
0 & \text{otherwise.}
\end{cases}
\end{equation}

This condition has the drawback that it applies only to representations into adjointable operators on Hilbert modules, but has the advantage that it works for arbitrary product systems which are not necessarily compactly aligned. With the advent of compactly aligned systems, condition (\ref{D:NC2}) was replaced by (\ref{D:NC1}). Both conditions are equivalent for concrete representations by \cite[Proposition 5.6]{Fow02}. Finally recall that in the case where $(G,P)$ is an abelian, lattice ordered group, we always have $p\vee q<\infty$ for any $p,q \in P$, so that the formulas simplify in both (\ref{D:NC1}) and (\ref{D:NC2}).

Each product system $X$ has a natural Nica-covariant isometric representation on Fock space which we now describe. We denote by $\cF_X := \oplus_{p \in P}X_p$ the direct sum of sequences. We then define $l : X \rightarrow \cL(\cF_X)$ given by $l_p(x)(y_q)_{q\in P} = (xy_q)_{q\in P}$ for each $p\in P$, $x\in X_p$, and $(y_q)_{q\in P} \in \cF_X$. We call $l$ the Fock representation, which is an isometric Nica-covariant representation of $X$ by \cite[Lemma 5.3]{Fow02}.

We denote by $\cN \cT_X$ the universal C*-algebra generated by a Nica-co\-variant representation for $X$, which exists due to \cite[Theorem 6.3]{Fow02}. Hence, there is an isometric Nica-covariant representation $i_X : X \rightarrow \cN \cT_X$ such that $\cN \cT_X$ is generated by the image of $i_{X}$ and for any other isometric Nica-covariant representation $\psi : X \rightarrow B(\cH)$ there exists a $*$-homomorphism $\psi_* : \cN \cT_X \rightarrow B(\cH)$ such that $\psi_* \circ i_{X,p} = \psi_p$, for every $p\in P$.

In \cite{CLSV11} the authors introduce a pair $(\N\O_X^r, j)$ which is co-universal for isometric, Nica-covariant, gauge compatible representations of X in the following sense: $\N\O_X^r$ is a $\ca$-algebra and $j : X \rightarrow \N\O_X^r$ is a Nica-covariant representation satisfying the following properties
\begin{enumerate}
\item 
$j_e$ is faithful,
\item
$j_*$ is gauge compatible surjection, where $j_* : \cN \cT_X \rightarrow \N\O_X^r$ is the canonical $*$-homomorphism induced by $j$, and
\item for any gauge-compatible Nica-covariant isometric representation $\psi : X \rightarrow \B$ for which $\psi_e$ is faithful, there is a surjective $*$-homomor- phism $q : \ca (\{\psi_p(X_p)\}_{p \in P}) \rightarrow \N\O_X^r$ such that $$q \circ \psi_p(\xi) = j_p(\xi), \mbox{ for all } \xi \in X_p \mbox{ and } p \in P.$$
\end{enumerate}

The existence of the pair $(\N\O_X^r, j)$ is not guaranteed a priori by its defining properties; nevertheless  if it exists then it is unique \cite[Theorem 4.1]{CLSV11}. Most of \cite{CLSV11} was devoted to showing that in many cases, such a pair does exist. However various other cases were left open in \cite{CLSV11}; in particular, the case for product systems over abelian, lattice ordered groups was not completely resolved. This was finally resolved in \cite{DorK}; see Theorem~\ref{T:CNP-envelope} below.

\subsection{Operator algebras and $\ca$-envelopes}

Let $\cA$ be an operator algebra. We say that the pair $(\B, \iota)$ is a {\em C*-cover} for $\cA$, if $\iota :\cA \rightarrow \cB$ is a completely isometric homomorphism, and $\ca(\iota(\cA)) = \cB$.

There is always a unique, smallest C*-cover for an operator algebra $\cA$. This C*-cover $(C_{e}^*(\cA), \kappa)$ is called the {\em C*-envelope} of $\cA$ and it satisfies the following universal property: given any other C*-cover $(\cB,\iota)$ for $\cA$, there exists a (necessarily unique and surjective) $*$-homomorphism $\pi:\cB \rightarrow C_{e}^*(\cA)$, such that $\pi \circ \iota = \kappa$. We will sometimes identify $\cA$ with its image $\iota(\cA)$ under a given C*-cover $(\B, \iota)$ for $\cA$. We say that $\rho : \A \rightarrow B(\H)$ is a \emph{representation} of $\A$ if $\rho$ is a completely contractive homomorphism. We refer the reader to \cite{Kat1} for a gentle introduction or \cite{BLM04, Paulsen} for a more in-depth treatment of the theory.

Let us consider identifying the $\ca$-envelope of a concrete operator algebra that plays an important role in the theory of product systems. Let $X$ be a product system $X$ over a (discrete) abelian, lattice ordered group $(G,P)$. The \textit{Nica tensor algebra} of the product system $X$ is given by
$$
\cN \cT^+_X := \overline{\alg}^{\| \cdot \|} \{ (i_{X})_p(X_p) \mid p \in P\}.
$$
The tensor algebra $\N\T^+_X$ is completely isometrically isomorphic to the non-selfad\-joint operator algebra generated by the image of the Fock representation $l \colon X \rightarrow \L(\F_X)$. It is naturally a subalgebra of $\cN \cT_X$, and is also the universal norm-closed operator algebra generated by a Nica-covariant isometric representation of $X$. 

The following is the central result of \cite{DorK} and among others it guarantees the existence of $\N\O_X^r$ for product systems over abelian lattice orders. It also extends an earlier result of the author and Kribs~\cite{KK06b}.

\begin{theorem} \label{T:CNP-envelope}
Let $X=\{X_p\}_{p \in P}$ be a compactly aligned product system over an abelian lattice order $(G, P)$. Then the $\ca$-algebras
\begin{itemize}
\item[\textup{(i)}] $\ca_e(\N\T_X^+)$, the $\ca$-envelope of the Nica tensor algebra $\N\T^+_X$,
\item[\textup{(ii)}] $ \cN \cO_X^r$, the co-universal $\ca$-algebra for gauge-compatible, Nica covariant representations of X, of Carlsen, Larsen, Sims and Vittadello \cite{CLSV11}, and
\item[\textup{(iii)}] Sehnem's $\ca$-algebra $X_e \times_XP$ \cite{Seh+},
\end{itemize}are mutually isomorphic via maps that send generators to generators.
\end{theorem}

\subsection{Crossed products of $\ca$-algebras}
Let $(\A, \G, \alpha)$ be a $\ca$-dynamical system and $\pi\colon \A \rightarrow B(\H)$ a representation. If $\tilde{\H}=\H \otimes L^2(\G)$, then we have a representation $\tilde{\pi}: \A \rightarrow B(\tilde{\H})$, with 
\[
\big(\tilde{\pi}(a)f\big)(\xi)=\pi(\alpha_{\xi^{-1}}(a))f(\xi), \mbox{ where } f \in L^2(\G, \H), \xi \in \G,
\]
and we also also the inflated left regular representation $\lambda\otimes I$ of $\G$. The pair $(\tilde{\pi}, \lambda\otimes I)$ forms a covariant representation of the dynamical system $(\A, \G, \alpha)$ and induces a representation 
\[
\pi_{\#}: \A\cpf \longrightarrow B(\tilde{\H});
\]
this representation is called the regular representation induced by $\pi$.

Let $(\A, \G, \alpha)$ be a $\ca$-dynamical system and $\pi\colon \A \rightarrow B(\H)$ a representation. It is well-known that $\pi$ is non-degenerate if and only if $\pi_{\#}$ is.  Proposition~\ref{l;elem} below elaborates on that theme. But first we need the following.

\begin{lemma} \label{l;elem}
Let $(\A, \G, \alpha)$ be a $\ca$-dynamical system and let $Y \subseteq \A$. If $(\pi, U, \H)$ is a covariant representation of $(\A, \G, \alpha)$, then,
\[
[(\pi\rtimes U)(Y\cpf)\H ]= [\pi(Y)\H]
\]
\end{lemma}

\begin{proof}
We need only to verify that 
\[
 [\pi(Y)\H] \subseteq [(\pi\rtimes U)(Y\cpf)\H ]
\]
as the other inclusion is trivial. Let $y \in Y$, $h \in \H$ and $\epsilon >0$. Let $\U$ be a neighborhoud of the identity in $\G$ so that 
\[
\|(U_{\zeta} -I)h\|\leq \epsilon, \mbox{ for all } \zeta \in \U.
\]
 Consider a positive function $g \in C_c(\G)$ with support contained in $\U$ so that $\int g(\zeta)d\zeta= 1$ and let $f \in C_c(\G, Y)$ with $f(\zeta) := g(\zeta)y$, $\zeta \in \G$. Then,
 \begin{align*}
| \langle (\pi\rtimes U)(f)h -\pi(y)h\mid k\rangle|&=\Big| \Big\langle \left( \int g(\zeta) \pi(y)U_{\zeta}d\zeta  - \int g(\zeta) \pi(y)d\zeta \right)h \mid k\Big\rangle\Big|\\
& \leq \int |\langle\pi(y)(U_{\zeta}-I)h \mid k\rangle| g(\zeta)d\zeta \leq \epsilon \|y\|,
 \end{align*}
 for any unit vector $k \in \H$. This estimate suffices to prove the lemma.
\end{proof}

\begin{proposition} \label{l;elem}
Let $(\A, \G, \alpha)$ be a $\ca$-dynamical system and let $\pi \colon \A\rightarrow B(\H)$ inducing a representation $\pi_{\#}\colon \A\cpf \rightarrow B(\tilde{\H})$. If $Y\subseteq \A$ is $\alpha$-invariant, then 
 \[
[\pi_{\#}(Y\cpf)\tilde{\H}]= L^2(\G, [\pi(Y)\H]).
\]
\end{proposition}

  \begin{proof}
In view of Lemma~\ref{l;elem} we need to prove that
 \[
[\tilde{\pi}(Y)\tilde{\H}] = L^2(\G, [\pi(Y)\H]).
 \]
Towards this end notice that $L^2(\G, [\pi(Y)\H])$ is generated by functions of the form
\begin{equation}\label{e:kappa}
\kappa(\zeta) = \pi(f(\zeta))h, \, \zeta \in \G,
\end{equation}
where $ f \in C_c(\G, Y)$ and $h \in \H$. Indeed, $L^2(\G,  [\pi(Y)\H])$ is generated as a Hilbert space by functions of the form 
\[
\lambda(\zeta) = g(\zeta)\pi(y)h= \pi(g(\zeta)y)h, \, \zeta \in \G,
\]
where $g \in C_c(G)$, $y \in Y$ and $h \in \H$. But such functions are as in (\ref{e:kappa}).

Consider now a function $\kappa$ as in (\ref{e:kappa}). For any $ \zeta \in \G$, there exists a neighborhood $\U_{\zeta}$ of $\zeta$ so that 
\begin{equation} \label{e;ah}
\| f(\zeta) - \alpha_{\xi^{-1}} \big(\alpha_{\zeta }(f(\zeta))\big)\|= \|\alpha_{\zeta^{-1}}\big(\alpha_{\zeta}(f(\zeta))\big)- \alpha_{\xi^{-1}} \big(\alpha_{\zeta }(f(\zeta))\big)\|\leq \epsilon
\end{equation}
and
\begin{equation}\label{e;vah}
\|f(\zeta) -f(\xi)\|\leq \epsilon, 
\end{equation}
for all $\xi \in \U_{\zeta}$. The collection $\{ \U_{\zeta}\}_{\zeta\in \G}$ forms an open cover for $\supp f$ (= the support of $f$) and so we obtain an open subcover $\U_{\zeta_1}, \U_{\zeta_2}, \dots \U_{\zeta_n}$. Create a Borel collection of disjoint sets $E_1, E_2, \dots , E_n$ so that 
$$ \cup_{i=1}^nE_i=\cup_{i=1}^n \U_{\zeta_i}$$
and $E_i \subseteq \U_{\zeta_i}$, for all $i = 1, 2, \dots , n$. Consider the function 
\[
\kappa_0:=\sum_{i=1}^n \tilde{\pi}\big( \alpha_{\zeta_i}(f(\zeta_i))\big)(\chi_{E_i}h) \in \tilde{\pi}(Y) \tilde{\H}.
\]
Using (\ref{e;ah}) and (\ref{e;vah}) we see that poinlwise the function $\kappa_0$ is $2\epsilon \|h\|$-close to $\kappa$ and so in the $L^2$-norm the two functions are $2\epsilon \|h\|\mu(\supp f)$-close. This estimate suffices to complete the proof.
\end{proof}


\section{The Hao-Ng isomorphism for product systems and abelian gauge actions} \label{s;HaoNg}

If $\G$ is a locally compact group acting on a $\ca$-correspondence $(X, \C)$, then the Hao-Ng Theorem \cite[Theorem 2.10]{HN} asserts that
\[
\O_X \cpf \simeq \O_{X\cpf} 
\]
provided that $\G$ is an amenable locally compact group. This result is having an increasing impact on current $\ca$-algebra research as witnessed in \cite{Ab, Deaconu, DKQ, Sch}. The \textit{Hao-Ng isomorphism problem} asks whether the above isomorphism remains valid, for either the full or the reduced crossed product, if one moves beyond the class of amenable groups. This problem is currently under investigation by several $\ca$-algebraists \cite{BKQR, KQR, KQR2, KR16, KR}.

In this paper we study a generalization of the Hao-Ng isomorphism first introduced in \cite{DorK}. Let $X=\{X_p\}_{p \in P}$ be a product system over a quasi-lattice order $(G, P)$. An action $\alpha: \G \rightarrow \Aut \N\T_X$ of a locally compact group is said to be a generalized gauge action if $\alpha_s(X_p)\subseteq X_p$, for all $s \in \G$ and $p \in P$. (In what follows we always identify $X$ with its image inside either $\N\T_X$ or $\N\O_X^r$, whatever convenient.) Such an action allows us to consider a new product system $X \cpf $ over $(G, P)$ defined as follows. 

For each $p \in P$, let $X_p\cpr$ be the closed subspace of $\N\T_X \cpr$ generated by $C_c(\G , X_p) \subseteq \N\T_X\cpr$. Just as in \cite[Lemma 7.11]{KR16}, one can verify that for every $p \in P$ we have that $X_p\cpr$ is an $X_e \cpr$-correspondence with inner product defined by $\lip f,g\rip =f^*g$ for $f , g \in X_p \cpf\subseteq\N\T_X \cpr$. Furthermore, it is easily seen that $(X_p \cpr)(X_q \cpr) \subseteq X_{pq} \cpr$ is dense for any $p,q \in P$. (Most of these claims can actually be proven using ideas appearing in Lemma~\ref{l;inclusion} and Remark~\ref{r;inclusion}.) Therefore $\{ X_p\cpr \}_{p \in P}$ forms a product system over $(G, P)$ that we denote as $X \cpr$. The Hao-Ng isomorphism problem asks whether 
\begin{equation} \label{e;HaoNg}
\N\O^r_X\cpr \simeq \N\O^r_{X\cpr},
\end{equation}
with a similar problem also holding for the full crossed product. The main goal of this section is to resolve (\ref{e;HaoNg}) in the case where both $G$ and $\G$ are abelian. As a consequence we obtain a generalization of Takai duality in the context of product systems over abelian lattice orders. In the case of $\ca$-correspondences this was first done by Abadie~\cite{Ab}.

The $\ca$-algebra $ \N\T_X\cpr $ also contains a non-selfadjoint crossed product algebra which we denote as $\N\T^+_X\cpr$, which is the norm closed algebra generated by $X\cpr$. Therefore, we have the inclusions
\[
X\cpr \subseteq \N\T^+_X\cpr\subseteq \N\T_X \cpr.
\]
and $\N\T_X \cpr$ is a $\ca$-cover for $\N\T^+_X\cpr$.

\begin{remark}\label{R:clarify copy}
(i) The reader familiar with the theory of non-selfadjoint crossed products (as developed in \cite{KR16}) recognizes that $\N\T^+_X\cpr$ coincides with the reduced crossed product of the dynamical system $( \N\T_X^+, \G, \alpha)$, as defined in \cite[Definition 3.17]{KR16}. Indeed this follows from an immediate application of \cite[Corollary 3.15]{KR16}.

(ii) When $G$ is abelian, we have that $\N\T_X^+\cpr$ is completely isometrically isomorphic to the natural subalgebra of $\N\O^r_X \cpr$ generated by all $C_0(G, X_p) \subseteq \N\O^r_X \cpr$, $p \in P$. This follows from Theorem \ref{T:CNP-envelope} and \cite[Corollary 3.16]{KR16}. Hence, we also obtain an injective copy of $X\cpr$ sitting naturally inside $\N\O^r_X\cpr$.
\end{remark} 

In order to establish the Hao-Ng isomorphism in our context, we need to provide a workable criterion for verifying the compact alignment of product systems. This is done in Proposition~\ref{p;criterion}. The lack of such a criterion was an impediment in \cite{DorK} that forced us to deal with actions of discrete groups only. We start by establishing some useful facts. 

Let $X=\{X_p\}_{p \in P}$ be a product system over $P$ and $\psi\colon X \rightarrow B(\H)$ a Nica-covariant representation. Recall from Section~\ref{ss;NT} that for each $p \in P$ we have a representation
\[
\psi_p\colon  \K(X_p) \longrightarrow B(\H); \theta_{x, y}\longmapsto \psi_p(x)\psi_p(y)^*, \,\, x , y \in X_p
\]
and its strict-sot continuous extension $\rho_p\colon \L(X_p)\rightarrow B(\H)$. We have the identities
\[
\rho_p(S) \psi_p(x)h= \psi_p(Sx)h, \mbox{ and}
\]
\[
\rho_{pq}(S\otimes I)=\rho_p(S)\rho_{pq}(I)=\rho_{pq}(I)\rho_p (S),
\]
for any $S \in \L(X_p)$ $x \in X_p$ and $h \in \H$. 
In general, $\rho_p(S)$ commutes with any $\rho_q(I)$, provided that $p\leq q$. Also
\begin{align*}
\psi_{pq}\big((S\otimes I)x\big) &= \rho_{pq}(S\otimes I)\psi_{pq}(x)\\
					&=\rho_p(S)\rho_{pq}(I)\psi_{pq}(x)\\
					&=\rho_p(S)\psi_{pq}(x).
\end{align*}
Hence if $S\in \K(X_p)$ and $T\in \K(X_q)$, then 
\[
\rho_{p\vee q}\big( (S\otimes I)(T \otimes I)x\big)=\rho_p (S)\rho_q (T)\psi_{p \vee q}(x) 
			=\psi_p (S)\psi_q(T)\psi_{p \vee q}(x)
\]
and so $\rho_{p\vee q}\big( (S\otimes I)(T \otimes I)\big)=\psi_p (S)\psi_q(T)\rho_{p \vee q}(I)$. Since we have assumed that $\psi$ is Nica-covariant, we have
\begin{align*}
\psi_p (S)\psi_q(T)\rho_{p \vee q}(I) &= \psi_p(S) \rho_{p \vee q}(I) \psi_q(T) \\
		&=\psi_p(S) \rho_p (I) \rho_q(I)\psi_q(T) \\
		&=\psi_p(S) \psi_q(T)
\end{align*}
and therefore
\[
\rho_{p\vee q}\big( (S\otimes I)(T \otimes I)\big)=\psi_p (S)\psi_q(T).
\]
This last identity plays an important role in the following

\begin{proposition} \label{p;criterion}
Let $X=\{X_p\}_{p \in P}$ be a product system over a lattice ordered abelian group $(G, P)$. Let $\psi\colon X \rightarrow B(\H)$ a faithful Nica-covariant representation that satisfies the following property: for any $s, t \in P$ and $x \in X_p$, $y \in X_q$ we have,
\[
\psi_p(x)^*\psi_q(y)\in \overline{\spn}\{\psi (z)\psi(w)^*\mid z \in X_{p^{-1}(p \vee q)}, w \in X_{q^{-1}(p \vee q)}\}.
\]
Then $X$ is compactly aligned. 

Conversely, if $X$ is compactly aligned, then any representation of $X$ satisfies the above property.
\end{proposition}
\begin{proof}
Assume that $\psi\colon X \rightarrow B(\H)$ satisfies the stated property and let $x, x'\in X_p$, $y,y'\in X_q$. Then 
\[
\psi_p(x')\psi_p(x)^*\psi_q(y)\psi_q(y')^* \in X_{p \vee q}X_{p \vee q}^*=\psi_{p\vee q}\big(\K(X_{p\vee q})\big).
\]
Let $K \in \K(X_{p\vee q})$ so that $\psi_{p \vee q}(K) = \psi_p(x')\psi_p(x)^*\psi_q(y)\psi_q(y')^*$. Then for any $ \xi \in X_{p \vee q}$ we have
\begin{align*}
\psi_{p\vee q}\big( (\theta_{x',x}\otimes I)(\theta_{y, y'}\otimes I)\xi \big) &=\rho_{p\vee q}\big( (\theta_{x',x}\otimes I)(\theta_{y, y'}\otimes I) \big)\psi_{p\vee q}(\xi ) \\
&=\psi_p(\theta_{x',x})\psi_q(\theta_{y,y'})\psi_{p\vee q}(\xi )  \\
&=\psi_p(x')\psi_p(x)^*\psi_q(y)\psi_q(y')^* \psi_{p\vee q}(\xi )  \\
&= \psi_{p \vee q}(K\xi).
\end{align*}
Since $\psi$ is faithful we obtain that $(\theta_{x',x}\otimes I)(\theta_{y, y'}\otimes I) = K \in \K(X_{p\vee q})$. This suffices to show that $X$ is compactly aligned.

The converse is proven in \cite[Proposition 5.10]{Fow02}.
\end{proof}

We will use this criterion to prove the permanence of compact alignment under crossed products by gauge actions. First, we need a large supply of Nica-covariant representations for $X\cpr$.

\begin{lemma} \label{L:supply}
 Let $(G,P)$ be an abelian, lattice ordered group and $X=\{X_p\}_{p \in P}$ a product system over $P$. Let $\alpha : \G \rightarrow \Aut \N\T_X$ be a generalized gauge action by a locally compact group $\G$. If $\psi \colon \cN\cO^r_X \rightarrow B(\H)$ is a $*$-representation, then the restriction of the regular representation 
 \[
 \psi_{\#} \colon \cN\cO_X \cpr \rightarrow B(\H \otimes L^2(\G))
 \]
 on the product system $X\cpr \subseteq \cN\cO_X^r \cpr$ forms a Nica-covariant representation of $X\cpr$.
 \end{lemma}
 
 \begin{proof}
 Let $s, t \in P$. By Proposition~\ref{l;elem} we have
 \begin{align*}
 [\psi_{\#}(X_p\cpr)\tilde{\H}] \cap  [\psi_{\#}(X_q\cpr)\tilde{\H}] &=L^2(\G, [\psi(X_q)]) \cap L^2(\G, [\psi(X_q)])\\
 &=L^2(\G, [\psi(X_p) ] \cap [\psi(X_q)])\\
 &=L^2(\G, [\psi(X_{p\vee q})])\\
 &=  [\psi_{\#}(X_{p \vee q}\cpr)\tilde{\H}] 
 \end{align*}
 and the conclusion follows.
 \end{proof}
 
To prove the permanence of compact alignment we also need the following.

\begin{lemma} \label{l;inclusion}
Let $(G,P)$ be an abelian, lattice ordered group and $X=\{X_p\}_{p \in P}$ a product system over $P$. Let $\alpha : \G \rightarrow \N\T_X$ be a generalized gauge action by a locally compact group $\G$. Then, for any $p,q \in P$, we have 
\[
C_c (\G, X_pX_q^*)\subseteq C_c(\G, X_p)C_c(\G, X_q)^*,
\]
when considered as subsets of $\N\O^r_X\cpr$.
\end{lemma}

\begin{proof}
Let $f \in C_c (\G, X_pX_q^*) \subseteq \N\O^r_X\cpr$ of the form $f=x_px_q^*g$, where $x_p \in X_p$, $x_q \in X_q$ and $g \in C_c(\G)$. Linear combinations of such elements form a dense subset of $C_c (\G, X_pX_q^*)$; hence it suffices to examine whether an $f$ of the above form belongs to $ C_c(\G, X_p)C_c(\G, X_q)^*$.

Towards this end, recall that $X$ is non-degenerate. Therefore $\N\O^r_X$ admits an approximate unit from $X_e$ and so $\N\O^r_X\cpr$ admits an approximate unit $\{e_i\}_{i \in \bbI}$ with $e_i \in C_c (\G, X_e)$. Then in the multiplier algebra of $\N\O^r_X\cpr$ we have,
\[
f = \lim_i x_p\big( e_i(x_q^*g)\big)=\lim_i(x_pe_i)(x_q^*g).
\]
However, it is easy to see that $x_pe_i\in C_c(\G, X_p)$. Also let $h\in C_c(\G, X_q)$ with $h(\xi) = \alpha_{\xi}(x_q)\Delta(\xi)\overline{g(\xi^{-1})}$, $ \xi \in \G$, where $\Delta$ is the modular function on $\G$ with respect to the Haar measure. Then $h^*=  x_q^*g$ and so $x_q^*g \in C_c(\G, X_q)^*$. The conclusion now follows.
\end{proof}

\begin{remark}  \label{r;inclusion}
Actually, Lemma \ref{l;inclusion} can be strengthened to show that 
\[
X_pX_q^*\cpr = (X_p\cpr) (X_p\cpr)^*, \,\, p,q \in P,
\]
and similarly
\[
X_p^*X_q\cpr = (X_p\cpr)^* (X_p\cpr), \,\, p,q \in P.
\]
\end{remark}

 \begin{proposition}
Let $(G,P)$ be an abelian, lattice ordered group and $X= \{X_p\}_{p \in P}$ a product system over $P$. Let $\alpha : \G \rightarrow \Aut \N\T_X$ be a generalized gauge action by a locally compact group $\G$. If $X$ is compactly aligned, then $X \cpr$ is also compactly aligned.
\end{proposition}

\begin{proof}
Let  $\psi:\N\O_X^r\rightarrow B(\H)$ be a faithful representation and let $p,q\in P$. Then, by Lemma~\ref{l;inclusion} and the remark above we have 
\begin{align*}
\psi_{\#}(X_p\rtimes_{\alpha}&\G)^*\psi_{\#}(X_p\rtimes_{\alpha}\G) = \psi_{\#}(X_p^*X_q\cpr)  \\
&= \psi_{\#}\big( [ f \in C_c(\G, \N\O_X^r)\mid f(\xi) \in X_p^*X_q,  \forall \xi \in \G]\big) \\
&\subseteq \psi_{\#}\big( [ f \in C_c(\G, \N\O_X^r)\mid f(\xi) \in X_{p^{-1}(p\vee q)}X_{q^{-1}(p\vee q)}^* \forall \xi \in \G ]\big)\\
&=  \psi_{\#}(X_{p^{-1}(p\vee q)} \cpr)   \psi_{\#}(X_{q^{-1}(p\vee q)} \cpr)^* .
\end{align*}
Apply Proposition~\ref{p;criterion} now with $X\cpr$ in place of $X$ to deduce that $X\cpr$ is compactly aligned.
\end{proof}

In the case where $X=\{X_p\}_{p \in P}$ is product system over an abelian total order $(G, P)$, the following result was essentially proven in \cite{KR}; see \cite[Theorem 3.7]{KR} and \cite[Remark 3.8]{KR}. It was also noted there that if certain issues regarding compact alignment and Nica covariance were resolved, then the result would hold for arbitrary quasi-lattice orders. These issues have been just resolved for abelian lattice orders and so we have 

\begin{theorem} \label{HNtensor2}
Let $X=\{X_p\}_{p \in P}$ a compactly aligned product system over an abelian lattice order $(G, P)$. Let $\alpha : \G \rightarrow \Aut \N\T_X$ be a generalized gauge action by a locally compact group $\G$. Then
\[
\cN\cT^{+}_X \cpr \simeq \cN\cT^+_{X\cpr}.
\]
Therefore, we have $$\cenv\big (\cN\cT^+_{X} \cpr \big) \simeq  \N\O^r_{X\cpr}.$$
\end{theorem}

We may now use the above as in our earlier works \cite{DorK, KatsIMRN, KR16} to obtain our extension of the Hao-Ng isomorphism. 

\begin{theorem} \label{HaoNgreduced}
Let $(G,P)$ be an abelian, lattice ordered group and let $X$ be a compactly aligned product system over $P$. Let $\alpha : \G \rightarrow \Aut \N\T_X$ be a generalized gauge action by a locally compact abelian group $\G$. Then
\begin{equation} \label{eq;Hao-Ng}
\N\O^r_X \cpr \simeq\N\O^r_{X\cpr}.
\end{equation}
  \end{theorem}
    
  \begin{proof} 
  By Theorem~\ref{HNtensor2} we have $$\cenv\big (\cN\cT^+_X \cpr \big) \simeq  \N\O^r_{X\cpr}.$$ On the other hand \cite[Theorem 3.23]{KR16} implies that $$\cenv\big (\cN\cT^+_X\cpr \big) \simeq \cenv (\cN\cT^+_X) \cpr .$$ Hence $\cenv (\cN\cT^+_X) \cpr  \simeq \N\O^r_{X\cpr}$. Now an application of Theorem~\ref{T:CNP-envelope} shows that $\cenv (\cN\cT^+_X) \simeq \N\O^r_X$ via a $\G$-equivariant map that intertwines the corresponding generalized gauge actions and the conclusion follows.
 \end{proof}
 
  \begin{remark}
The only reason why we state Theorem~\ref{HaoNgreduced} for  generalized gauge actions by abelian groups instead of arbitrary ones is that currently given an arbitrary operator algebra $\A$, the identity 
\begin{equation} \label{eq;cenv}
\cenv(\A\cpr)\simeq \cenv(\A)\cpr
\end{equation}
is known to hold only for either discrete groups \cite{KatsIMRN} or locally compact abelian groups~\cite[Theorem 3.23]{KR16}. Note however that if one assumes that $\A$ is hyperrigid, then (\ref{eq;cenv}) holds for any locally compact group~\cite[Theorem 3.6]{KR}. In subsequent work \cite{KRfuture} we show that the tensor algebra $\N\O_X^r$ of a regular product system $X$ over an abelian lattice order $(G, P)$ is actually hyperrigid and therefore the Hao-Ng isomorphism (\ref{eq;Hao-Ng}) holds for any regular product system $X$ and any locally compact group $G$.
 \end{remark} 

\section{Crossed Products by the dual action and Takai Duality}
Let $X = \{X_p\}_{p \in P}$ be a product system over a lattice ordered, abelian group $(G, P)$. 
We are now defining a product system $C_0(G, X)$ with fibers $C_0(G, X_p)$
consisting of $X_p$-valued nets, indexed by $G$ and ``vanishing at infinity". Each $C_0(G, X_p)$ becomes a $\ca$-correspondence over $C_0(G, X_e)$ by
\[
(\phi f)(s) = \phi (s) f(s), \, (f\phi )(s)= f(s)\phi(p^{-1}s),
\]
\[
\quad \langle f, g\rangle (s) =f(p s)^* g(ps), 
\]
for $s \in G$, $\phi \in C_0(G, X_e)$ and $f, g \in C_0(G, X_p)$. If $X_p$ happens to be a Hilbert bimodule, then $C_0(G, X_p)$ becomes a Hilbert bimodule as well with left inner product
\[
[ f, g](s) = f(s)g(s)^*, \,\, s \in G, f, g \in C_0(G, X_p).
\]
Note that in each fiber $C_0(G, X_p)$ we have a class of distinguished elements, the coordinate functions  $x_p\delta_s$ defined by $x_p\delta_s(t)=x_p$, if $t=s$ and $0$ otherwise. For different $s \in G$ these elements are orthogonal with respect to both inner products; put together they generate $C_0(G, X_p)$ as a Banach space.

For each pair $p, q \in P$ we define 
\[
U_{p,q}\colon C_0(G, X_p)\otimes C_0(G, X_q)\longmapsto C_0(G, X_{pq})
\]
by
\begin{equation} \label{def;mult}
U_{p, q}(x_p\delta_s\otimes x_q\delta_t)=
\begin{cases}
x_px_q\delta_{s}, &\mbox{ if } t =p^{-1}s\\
			0, &\mbox{otherwise.}
\end{cases}
\end{equation}
(Note here that $x_p\delta_s\otimes x_q\delta_t =0$ if $ t \neq p^{-1}s$.) It is easy to see that $U_{p, q}$ extends to a well-defined operator on the linear span of elementary tensors in $C_0(G, X_p)\otimes C_0(G, X_q)$.

\begin{proposition}
Let $p, q \in P$ and $X$, $U_{p,q}$ as above. Then $U_{p,q}$ preserves inner products and therefore extends to an isometric bimodule operator from $ X_p\otimes X_q$ onto $X_{pq}$. Furthermore
\begin{equation} \label{eq;prodsyst}
U_{pq, q'}\circ(U_{p, q}\otimes I_{q'})= U_{p, qq'}\circ(I_p\otimes U_{q,q'}), \, \, q'\in P.
\end{equation}
Therefore $C_0(G, X)=\{C_0(G, X_p)\}_{p \in P}$ together with the family $\{U_{p,q}\}_{p,q \in P}$ becomes a product system.
\end{proposition}

\begin{proof} Indeed since
\[
\langle x_p\delta_s\otimes x_q\delta_t \mid x_{p}'\delta_{s'}\otimes x_{q}'\delta_{t'}\rangle =
\langle x_q\delta_t\mid\langle x_p\delta_s|x_p' \delta_{s'}\rangle x_q'\delta_{t'}\rangle\]
we have
\[
\langle x_p\delta_s\otimes x_q\delta_t \mid x_{p}'\delta_{s'}\otimes x_{q}'\delta_{t'}\rangle =
\begin{cases}
x_q^*x_p^*x_p'x_q', &\mbox{ if } s=s', t=t' \mbox{ and } s=pt\\
			0, &\mbox{ otherwise.}
\end{cases}
\]
The same estimate holds for $\langle U_{p,q}(x_p\delta_s\otimes x_q\delta_t ) \mid U_{p,q}(x_{p}'\delta_{s'}\otimes x_{q}'\delta_{t'})\rangle $ and so $U_{p,q}$ preserves inner products. Therefore it is bounded on the linear span of the elementary tensors of $ C_0(G, X_p)\otimes C_0(G, X_q)$ and so it extends to a (necessary isometric) surjection. 

Finally, one can similarly verify (\ref{eq;prodsyst}) on elementary tensors and the conclusion follows.
\end{proof}

On the Cuntz-Pimsner algebra of $C_0(G,X)$ we now define an action $\beta: G \rightarrow \Aut \N\O_{C_0(X,G)}^r$ as follows
\[
\beta_s \colon \N\O_{C_0(G,X)}^r\longrightarrow \N\O_{C_0(G,X)}^r; \sum_{p, t} x_p\delta_t\longmapsto \sum_{p, t} x_p\delta_{s^{-1}t}, \quad s \in G.
\]
We use now $C_0(G, X)$ to give a refinement of the Hao-Ng isomorphism in the case where the generalized gauge action on $X$ is actually coming the gauge action go $\widehat{G}$ itself. It allows us to give a picture for $\N\O^r_{X}\rtimes_{\alpha} \widehat{G}$ that does not involve the dual group $\widehat{G}$.

\begin{theorem} \label{thm;calc}
Let $X=\{X_p\}_{p \in P}$ be a product system over an abelian lattice order $(G, P)$. Then there exists a $*$-isomorphism
\[
\omega \colon \N\O^r_{X}\rtimes_{\alpha} \widehat{G} \longrightarrow  \N\O_{C_0(G,X)}^r
\]
which is equivariant with respect to the dynamical systems $(\N\O^r_{X}\rtimes_{\alpha} \Hat{G}, G, \widehat{\alpha})$ and $( \N\O_{C_0(X,G)}^r, G, \beta)$.
\end{theorem}

\begin{proof}
The outline of the proof is the following. We will show that the product systems $C_0(X, G)$ and $X\rtimes_{a}\widehat{G}$ are unitarily equivalent. The unitary equivalence will induce a $*$-isomorphism $\omega$ between $ \N\O_{C_0(X,G)}^r$ and $\N\O^r_{X\rtimes_{\alpha} \Hat{G} } \simeq \N\O^r_{X}\rtimes_{\alpha} \Hat{G} $ (Hao-Ng isomorphism; Theorem~\ref{HaoNgreduced}). The last assertion of the Theorem will follow by observing that $\omega$ conjugates $\widehat{\alpha}_s$ to $\beta_s$, for all $s \in G$.

The isomorphism between $X\rtimes_{a}\widehat{G}$ and $C_0(X, G)$ is established via the Fourier transform. If $f \in X_p \rtimes_{\alpha}\widehat{G}$ then we define $\hf \in C_0(G, X_p)$ by 
\[
\hf(s)=\int f(\zeta)\zeta(p^{-1}s)d\zeta, \, s \in G. 
\]
The verification that for each $p\in P$  the Fourier transform establishes a unitary equivalence between $X_p\rtimes_{a}\widehat{G}$ and $ C_0(G, X_p)$ depends on routine calculations. For instance, given $f, g \in X_p\rtimes_{a}\widehat{G}$, we have
\begin{align*}
\langle f \mid g \rangle^{\widehat{}}(s)&=\int\langle f \mid g\rangle(\zeta)\zeta(s)d\zeta =\int (f^*g)(\zeta)\zeta(s)d\zeta\\
&=\int\int \alpha_{\xi}(f(\xi^{-1})^*)\alpha_{\xi}(g(\xi^{-1}\zeta))\zeta(s)d\xi d\zeta\\
&=\int \int f(\xi^{-1})^*g(\zeta)\xi(s)\zeta(s) d\xi d\zeta\\
&=\Big(\int f(\xi)^*\overline{\xi(s)}d\xi\Big)\Big(\int g(\zeta)\zeta(s)d\zeta \Big)\\
&=\langle\hf(ps) \mid \hg(ps)\rangle= \langle \hf \mid\hg\rangle(s), \quad s \in G,
\end{align*}
which shows that the Fourier transform preserves the right inner product.

We also need to verify that the Fourier transform preserves the multiplication between the various fibers of $ X \rtimes_{a}\widehat{G}$. We will use the fact that linear span of the evaluation functions (characters)
\[
e_s\colon \widehat{G}\longrightarrow \bbC; \widehat{G} \ni\zeta \longmapsto \zeta(s), \quad s \in G,
\]
form a dense subset of $C(\widehat{G})$ by the Peter-Weyl Theorem.\footnote{These are the Fourier transforms of the coordinate functions $\delta_s$, $s \in G$, and satisfy $\widehat{e}_s=\delta_{s^{-1}}$, $s \in G$.} In light of this, it suffices to verify that 
\begin{equation} \label{eq;multfibers}
(x_pe_s\cdot x_qe_t)^{\widehat{}}= (x_pe_s)^{\widehat{}}\otimes (x_qe_t)^{\widehat{}}, \,\, s,t \in G.
\end{equation}
Towards this end, 
\begin{align*}
(x_pe_s\cdot x_qe_t)(\zeta)&= \int e_s(\xi) \alpha_{\xi}\left(x_qe_t(\xi^{-1}\zeta)\right) d\xi\\
&=x_px_q \zeta(t)\int \xi(qst^{-1})d\xi\\
&=x_px_q\zeta(t)=x_px_qe_t(\zeta),
\end{align*}
provided that $qt^{-1}=s^{-1}$, and $0$ in all other cases. Hence by taking Fourier transform in $X_{pq}\rtimes_{\alpha}\widehat{G}$ we obtain
\begin{equation} \label{eq;cases1}
(x_pe_s\cdot x_qe_t)^{\widehat{}}=
\begin{cases}
x_px_q\delta_{pqt^{-1}}=x_px_q\delta_{ps^{-1}}, &\mbox{ if } qt^{-1}=s^{-1}\\
0, &\mbox{otherwise}
\end{cases}
\end{equation}
On the other hand,
\[
(x_pe_s)^{\widehat{}} = x_p\delta_{ps^{-1}} \mbox{ and } (x_qe_t)^{\widehat{}}=  x_q\delta_{qt^{-1}}
\]
and so by (\ref{def;mult}) we obtain
\begin{equation} \label{eq;cases2}
x_p\delta_{ps^{-1}} \otimes x_q\delta_{qt^{-1}}=
\begin{cases}
x_px_q\delta_{ps^{-1}}, &\mbox{ if } qt^{-1}=p^{-1}ps^{-1} = s^{-1}\\
&\mbox{ otherwise}
\end{cases}
\end{equation}
By comparing (\ref{eq;cases1}) and (\ref{eq;cases2}), we obtain (\ref{eq;multfibers}) and so the Fourier transform preserves the product between fibers.

Finally, if $f \in X_p\rtimes_{\alpha} \widehat{G}$ then
\begin{align*}
(\widehat{\alpha}_t(f))^{\widehat{}}(s)&=\int \widehat{\alpha}_t(f)(\zeta) \zeta(p^{-1}s)d\zeta\\
&=\int\zeta(t)f(\zeta)\zeta(p^{-1}s)d\zeta\\
&=\int f(\zeta) \zeta(p^{-1}st)d \zeta \\
&=\hf(st)=\beta_t(\hf)(s), \, s \in G
\end{align*}
and the equivariance of $\omega$ has been established.
\end{proof}

In order to get more information on $\N\O_{C_0(X,G)}^r$ we focus on a special class of product systems. We start with product systems of equivalence bimodules over an abelian lattice order $(G, P)$. 

Recall that every $t \in G$ can be written in a most efficient way as $t = \sigma( t)q^{-1}$, where $\sigma( t)= t\vee e$ and $\tau(t)=t^{-1}\sigma(t)$. By efficient we mean that for any other decomposition $t = pq^{-1}$, $p, q \in P$ there exists $r \in P$ so that  $p=\sigma( t)r$ and $q = r\tau(t)$. 

For each $t \in G$ we define $X_t:=X_{\sigma(t)}X_{\tau(t)}^*$. It is easy to see that $X_t$ becomes an $X_e$-Hilbert $\ca$-module with the usual multiplication as right action and inner product defined by $\langle x_t | y_t\rangle = x_t^*y_t$, $x_t, y_t \in X_t$. 

\begin{lemma} \label{lem;full}
Let $X=\{X_p\}_{p \in P}$ be a product system of equivalence bimodules over an abelian lattice order $(G, P)$. Then $X_{s}X_{t}^*=X_{st^{-1}}$, for any $s,t \in G$.
\end{lemma}

\begin{proof}
Assume first that $s, t \in P$. Then there exists $r \in P$ so that  $s=\sigma(st^{-1})r$ and $t = r\tau(st^{-1})$. Hence, $$X_{s}X_{t}^*=X_{\sigma(st)}X_{r}X_{r}^*X_{\tau(st^{-1})}^*=X_{\sigma(st^{-1})}X_eX_{\tau(st^{-1})}^*= X_{st^{-1}}$$ since $X_r$ is full as a left module.

In order to establish the general case we first claim that for any $p, q \in P$ we have
\[
X_p^*X_q=X_{p^{-1}(p \vee q )}X_{q^{-1}(p \vee q )}^*.
\]
Indeed, since both $X_p$ and $X_q$ are full as right modules we have
\begin{align*}
X_{p^{-1}(p \vee q )}X_{q^{-1}(p \vee q )}^*&= X_p^*X_pX_{p^{-1}(p \vee q )}X_{q^{-1}(p \vee q )}^*X_q^*X_q\\
&=X_p^*X_{p \vee q }X_{p \vee q }^*X_q\\
&=X_p^*X_eX_q=X_p^*X_q,
\end{align*}
as desired.

Finally, consider $s,t \in G$ and let $p, p_1, q, q_1 \in P$ so that $s=pq^{-1}$ and $t =p_1q_1^{-1}$. Then by the previous paragraphs we have
\begin{align*}
X_sX_t^*&=X_pX_q^*X_{p_1}X_{q_1}^*\\
&=X_pX_{q^{-1}(q\vee p_1)}X_{p_1^{-1}(q \vee p_1)}^*X^*_{q_1}\\
&=X_{pq^{-1}(q\vee p_1)}X^*_{q_1p_1^{-1}(q \vee p_1)}.
\end{align*}
However, $pq^{-1}(q\vee p_1), q_1p_1^{-1}(q \vee p_1) \in P$ and so the first paragraph of the proof implies
\[
X_sX_t^*=X_{pq^{-1}(q_1p_1^{-1})^{-1}}=X_{st^{-1}}
\]
as desired.
\end{proof}

We denote with $\E_X$ the right $X_e$-Hilbert $\ca$-module defined by $\E_X=\oplus_{t \in G} X_t$. The $\ca$-module $\E_X$ and the compact operators acting on it play an important role in the proof of our Takai duality.  

Given $s, t \in G$ and $x_s \in X_s$ we define an operator $T_{x_s, t} \in \L(\E_X)$ by 
\[
T_{x_s, t}(x_re_r)=
\begin{cases}x_sx_re_t,& \mbox{ if }r=s^{-1}t \\
0, &\mbox{ otherwise,}
\end{cases}
\]
where $x_re_r$ denotes the element of $\E_X$ with $x_r\in X_r$ at the $r$-position and $0$ elsewhere. It is easy to see that $T_{x_s, t}$ is adjointable with adjoint $T_{x_s, t}^*=T_{x_s^*, s^{-1}t}$. Furthermore

\begin{lemma}
Let $X=\{X_p\}_{p \in P}$ be a product system of equivalence bimodules over an abelian lattice order $(G, P)$. Then the collection 
\[
\S_X=\{T_{x_p, t} \mid x_p \in X_p, p \in P, t \in G\} \subseteq \L(\E_X)
\]
generates $\K(\E_X)$ as a $\ca$-algebra.
\end{lemma}

\begin{proof}
First notice that 
\begin{equation} \label{eq;rankone}
\theta_{x_te_t, x_{s^{-1}t}e_{s^{-1}t}} =T_{x_tx^*_{s^{-1}t}, t},
\end{equation}
where $x_t \in X_t$, $x_{s^{-1}t} \in X_{s^{-1}t}$, $s, t \in G$. By Lemma~\ref{lem;full}, $X_t X_{s^{-1}t}^*=X_s$ and so (\ref{eq;rankone}) implies that all operators $T_{x_s, t}$, $x_s \in X_s$, $s, t \in G$, are compact. In particular, $\ca(\S_X) \subseteq \K(\E_X)$.

Furthermore
 \[
T_{x_p, t}(T_{x_q, qp^{-1}t})^*=T_{x_p, t}T_{x_q^*, p^{-1}t} = T_{x_px_q^*, t}
\]
and so $\ca(\S_X)$ contains all operators of the form $ T_{x_s, t}$, $x_s \in X_s$, $s,t \in G$. However, given any $x_s \in X_s$ and $x_t \in X_t$, $s, t \in G$, we have
\[
\theta_{x_se_s, x_te_t} = T_{x_sx_t^*, s}
\]
and so $\ca(\S_X)$ contains all rank-one operators. This suffices to prove the lemma.
\end{proof}

We have arrived to our first version of Takai duality that holds for Cuntz-Pimsner algebras  of product system of equivalence bimodules over an abelian lattice order. What the lattice order is just $(\bbZ, \bbN)$, this result was obtained by Abadie \cite{Ab}. Note that in Corollary~\ref{cor;TakaiDuality} we obtain a much more general result.

\begin{theorem}[Takai duality, first version] \label{mainTakai}
Let $X=\{X_p\}_{p \in P}$ be a product system of equivalence bimodules over an abelian lattice order $(G, P)$. If $\alpha$ is the dual action of $\widehat{G}$ on $\N\O_X^r $, then $\N\O_X^r \rtimes \widehat{G}$ is Morita equivalent to $X_e$.
\end{theorem}

\begin{proof}
For each $p \in P$ consider the map
\[
\psi_p \colon C_0(G, X_p)\longrightarrow \K(\E_X); \sum_{s \in G} x_{p_s} \delta_s \longmapsto \sum_{s \in G} T_{x_{p_s}, s}.
\]
Routine calculations show that $\psi_p$, $p \in P$, is a faithful Hilbert bimodule representation. For instance, given $x_p, y_p \in X_p$ and $s,t \in G$, we have
\begin{equation} \label{rightside}
\psi_e([x_p\delta_s | y_p\delta_t]) = 
\begin{cases}
\psi_e(x_py_p^*\delta_s) = T_{x_py_p^*,e}, &\mbox{ if } s=t \\
0, &\mbox{ otherwise.}
\end{cases}
\end{equation}
On the other hand, 
\[
[\psi_p(x_p\delta_s) | \psi_p(y_p\delta_t)] = T_{x_p, s}T_{y_p, t}^*=T_{x_p, s}T_{y_p^*, p^{-1}t}
\]
which equals the right side of (\ref{rightside}) and so $\psi_p$ preserves the left inner product. Similar calculations establish the other properties and also show that the family $\{\psi_p\}_{p \in P}$ preserves the product system structure.

Since each $\psi_p$ is a Hilbert bimodule (hence a Katsura) representation of $C_0(G, X_p)$, we have that $\psi$ is Cuntz-Pimsner covariant in the sense of Fowler \cite[Definition 2.5]{Fow02}. Corollary 5.2 in \cite{SY11} shows that $\psi$ is a Cuntz-Pimsner covariant representation \cite[Definition 3.9]{SY11}  and clearly gauge compatible. Hence by \cite[Corollary 4.8]{CLSV11}, it induces a faithful representation of $\N\O_X^r$ and so $\N\O_X^r\simeq \K(\E_X)$. The conclusion now follows.
\end{proof}

\begin{lemma} \label{dilationexst}
Let $X=\{X_p\}_{p \in P}$ be a product system of regular and full $\ca$-correspondences over an abelian lattice order $(G, P)$. Then there exists a product system $Y=\{Y_p\}_{p \in P}$ of equivalence bimodules so that $X_p \subseteq Y_p$, for all $ p \in P$, and satisfying $\N\O^r_X\simeq\N\O^r_Y$. 
\end{lemma}

\begin{proof}
Consider the copy of $X$ contained in $\N\O^r_X$ and let 
\[
Y_e= \ca( \{ X_sX_s^*\mid s\in P\}) \mbox{ and } Y_p = [  X_{p}Y_e, p \in P\backslash \{e\}].
\]
It is easily seen that each $Y_p$ equipped with the right $Y_e$-valued inner product 
\[
\langle\cdot, \cdot \rangle \colon Y_t\times Y_t\longrightarrow Y_e; (x'_p , x_p)\longmapsto \langle x'_p, x_p \rangle:= (x'_p)^*x_p
\]and the left $Y_e$-valued inner product 
\[
[\cdot, \cdot ] \colon Y_p\times Y_p\longrightarrow Y_e; (x'_p, x_p)\longmapsto [ x'_p, x_p]:= x'_p x_p^*
\]
becomes a Hilbert bimodule over $Y_e$. (The $\ca$-algebra $Y_e$ is usually referred to as the \textit{core} of  $\N\O^r_X$.)

For the proof of the lemma we need to verify first the following claim: 
\[
Y_e = [\cup_{p_0 \leq p} X_pX_p^*] ,
\]
for any $p_0 \in P$. Indeed, since $X_e$ acts faithfully by compacts, it coincides with the Katsura ideal for the correspondence $(X_{p_0}, X_e)$ and so $X_e \subseteq X_{p_0}X_{p_0}^*$. Therefore, for any $p \in P$ we have
\[
X_pX_p^*=X_pX_eX_p^* \subseteq X_p (X_{p_0}X_{p_0}^*)X_p^*\subseteq X_{pp_0}X_{pp_0}^*,
\]
and the claim follows.

Now for the proof of the lemma notice that for any $p_0\in P$ we have
\begin{align*}
Y_{p_0}Y_{p_0}^*&=X_{p_0}Y_eY_e^*X_{p_0}^*\\
&=[\cup_{p \in P} X_{p_0}X_pX_p^*X_{p_0}^*]\\
&=[\cup_{p_0\leq p } X_pX_p^*]\\
&=[\cup_{p\in P} X_p X_p^*]=Y_e,
\end{align*}
and so each $(Y_{p_0}, Y_e)$ is left-full. It is easily seen that each $(Y_{p_0}, Y_e)$ is right-full and so it is an equivalence bimodule.

We also need to verify that for any $p, q \in P$, we have $Y_pY_q=Y_{pq}$, as expected from $Y$ if it is to form a product system. Since $X_q$ is essential,
\[
Y_{pq} =X_{pq}Y_e =(X_pX_e)(X_qY_e) \subseteq Y_pY_q.
\]
Conversely, let $r,s \in P$. Then
\begin{align*}
X_p(X_rX_r^*)X_q(X_sX_s^*)&= X_{pr}(X_r^*X_{qs})X_s^*\\
                        &\subseteq X_{pr}(X_{r^{-1}(r \vee qs)}X^*_{(qs)^{-1}(r \vee qs) })X_s^* \\
			&= X_pX_{r\vee qs}X^*_{q^{-1}(r \vee qs)}\\
			&= X_pX_qX_{q^{-1}(r\vee qs)} X^*_{q^{-1}(r \vee qs)} \subseteq X_{pq}Y_e, 
			\end{align*}
where the first set theoretic inclusion above follows from Proposition~\ref{p;criterion}. This implies $ Y_pY_q.\subseteq Y_{pq}$ and so $Y$ forms a product system.

The fact that $\N\O^r_X\simeq\N\O^r_Y$ follows from considerations similar to that in the last paragraph of Theorem~\ref{mainTakai}. Indeed it is enough to show that the identity representation of $Y$ inside $\N\O^r_X$ is covariant\footnote{in the sense of Katsura} on each fiber and admits a gauge action. The covariance follows easily from the fact that the identity representation on each fiber preserves both inner products. The required gauge action is the one coming from the universality of $\N\O^r_X$.
\end{proof}

In \cite[Remark 5.4]{Sch} Schafhauser notes that in the case where $X$ is a regular $\ca$-correspondence over an approximately finite $\ca$-algebra then ``at least in principle" one could determine $\O_X\rtimes_{\alpha} \bbT$ up to Morita equivalence. (This is because in that case $\O_X\rtimes_{\alpha} \bbT$ is AF and his Theorem A calculates the K-theory of $\O_X\rtimes_{\alpha} \bbT$.) We now see that one can determine the crossed product of $\N\O_X^r \rtimes_{\alpha} \widehat{G}$ up to Morita equivalence in a much broader context and in a very precise manner. 

\begin{corollary}[Takai duality, second version] \label{cor;TakaiDuality}
Let $X=\{X_p\}_{p \in P}$ be a product system of regular and full $\ca$-correspondences over an abelian lattice order $(G, P)$. If $\alpha$ is the gauge action of $\widehat{G}$ on $\N\O_X^r $, then $\N\O_X^r \rtimes_{\alpha} \widehat{G}$ is Morita equivalent to the core $Y_e$ of $\N\O_X^r$.
\end{corollary}

\begin{proof}
Let $Y$ be the product system of the previous lemma and notice that the dual action $\alpha$ of $\widehat{G}$ on $\N\O_X^r$ coincides with that of $\widehat{G}$ on $\N\O_Y^r$. Therefore Theorem~\ref{mainTakai} implies that, 
\[
\N\O_X^r\rtimes_{\alpha}\widehat{G}\simeq \N\O_Y^r\rtimes_{\alpha}\widehat{G}\simeq_{m} Y_e
\]
and the conclusion follows.
\end{proof}

\section{Calculating the K-theory of $\N\O_X^r\rtimes_{\alpha}\widehat{G}$} \label{s;Sch}

In \cite{Sch} Schafhauser considers the question of when the Cuntz-Pimsner algebra $O_X$ of a $\ca$-correspondence $X$ is AF -embedable. A good part of his paper is occupied with the proof of his Theorem A that calculates the K-theory of the crossed product of $\O_X$ by the natural gauge action of the circle group $\bbT$. We now observe that his calculation generalizes to product systems over abelian lattice orders under the blanket assumptions of fullness and separability for all product systems and $\ca$-algebras involved. There are two key elements in the proof: Corollary~\ref{cor;TakaiDuality} and \cite[Theorem 3.22]{AM}.

It is well-known that an $\B$-$\A$ $\ca$-correspondence $H$ induces a group homomorphism $[H] : K_*(\B) \rightarrow K_*(\A)$. Furthermore, Schafhauser shows in \cite[Proposition 4.3]{Sch}, if $H$ is regular then $[H]$ is positive at the $K_0$-level. In our situation, if  $X=\{X_p\}_{p \in P}$ is a product system of regular and full $\ca$-correspondences over an abelian lattice order $(G, P)$, then we obtain a direct limit system $(\G_{p}, i_{p,q})_{p\leq q}$ of abelian groups $\G_{p}:= K_*(X_e)$ and connecting maps 
\[
i_{p,q}\colon K_*(X_e) \xrightarrow{\phantom{p}[X_{p^{-1}q}]\phantom{p}} K_*(X_e),
\] 
for any pair $(p, q)$, $p, q \in P$ with $p \leq q$. We denote with $\varinjlim(K_*(X), [X])$ the direct limit of the system $(\G_{p}, i_{p,q})_{p\leq q}$. See \cite[Theorem 3.22]{AM} and the discussion preceding that result for more information.

\begin{theorem}
Let $X=\{X_p\}_{p \in P}$ be a product system of regular $\ca$-corre\break spondences over an abelian lattice order $(G, P)$. If $\alpha$ is the dual action of $\widehat{G}$ on $\N\O_X^r $, then 
\[
K_*(\N\O_X^r \rtimes_{\alpha} \widehat{G})\simeq \varinjlim(K_*(X), [X]).
\]
\end{theorem}
\begin{proof}
By Corollary \ref{cor;TakaiDuality}, we need to calculate the K-theory of the core $Y_e$ of $\N\O_X^r$. This is done as follows.

For any pair $(p, q)$, $p, q \in P$ with $p \leq q$, let 
\[
\phi_{p,q} : \K(X_p)\longrightarrow \K(X_q)\colon T \longmapsto T\otimes I.
\]
Consider the directed system $(\K(X_p), \phi_{p,q})_{p\leq q}$. Then the direct limit $\ca$-algebra $\varinjlim (\K(X_p), \phi_{p,q})_{p\leq q}$ is the $\ca$-algebra $\O_1$ of Albandik and Meyer \cite[Lemma 3.13]{AM}. 

On the other hand, let $j: X\rightarrow \N\O_X^r$ be the universal Nica-covariant representation in the definition of $\N\O_X^r$. Since $X_e$ acts faithfully by compacts on the fibers of $X$, we have that $j_q\circ \phi_{p,q}\circ j_p^{-1}$ is simply the inclusion map of $X_pX_p^* \subseteq X_qX_q^*$, for any pair $(p, q)$, $p, q \in P$ with $p \leq q$. Therefore $j$ establishes an isomorphism between $\O_1$ and $Y_e$. 
Now the K-theory of $\O_1$ is shown by Albandik and Meyer \cite[second paragraph of the proof of Theorem 3.22]{AM} to be the exactly one appearing in the statement of the Theorem and the conclusion follows.
\end{proof}

\begin{remark}
In the case where $(G, P) = (\bbZ^k, \bbN^k)$, the previous result has been obtained in \cite{DL} without the assumption that $X$ consists of full $\ca$-correspondences. The author is grateful to the authors of \cite{DL} for sharing with him their results and also bringing to his attention the work of Albandik and Meyer \cite{AM}.
\end{remark}

It is also prudent to observe that Schafhauser's Theorem B also generalizes in our context. One needs to verify the following.

\begin{proposition}
Let $X=\{X_p\}_{p \in P}$ be a compactly aligned product system over an abelian lattice order $(G, P)$. If $X_e$ is an AF $\ca$-algebra then the core $Y_e$ of $\N\O_X^r$ is also AF.
\end{proposition}

\begin{proof}
As in \cite[Lemma 3.6]{CLSV11}, for a finite $\vee$-closed subset $F$ of $P$, consider 
\[
B_F:=\sum_{p \in F} \, X_pX_p^*\subseteq \N\O_X^r.
\]
We claim that $B_F$ is an AF $\ca$-algebra. We show this by induction on the size of $F$. If $F = \{p\}$ is a singleton, then $$B_F = X_pX_p^* \simeq \K(X_p).$$ Since $X_p$ is full, $B_F$ is Morita equivalent to $X_e$ and thus AF.

For the inductive step we proceed as in \cite[Lemma 3.6]{CLSV11}. Let $m \in F$ be a minimal element. Then by induction $B_{\{m\}}$ and $B_{F\backslash\{m\}}$ are AF $\ca$-algebras. Furthermore 
\[
B_F=B_{\{m\}} +B_{F\backslash\{m\}}
\] 
and $B_{F\backslash\{m\}}$ is is an ideal in $B_F$. Hence $B_F$ is closed and the extension of 
$$ B_{\{m\}} \slash (B_{\{m\}} \cap B_{F\backslash\{m\}})$$ by $B_{F\backslash\{m\}}$. Both of these $\ca$ algebras are AF and so $B_F$ is also AF and the claim follows.

In order to finish the proof observe that $Y_e$ is the direct limit of $\{ B_F\}$, where $F$ ranges over all finite $\vee$-closed subsets of $P$ and the conclusion follows.
\end{proof}

\begin{corollary}
Let $X=\{X_p\}_{p \in P}$ be a compactly aligned product system over an abelian lattice order $(G, P)$. If $X_e$ is an AF $\ca$-algebra then $\N\O_X^r \rtimes_{\alpha} \widehat{G}$ is also AF.
\end{corollary}

\begin{proof}
The proof follows from the previous proposition and \cite[Theorem 6.2]{Sch}.
\end{proof}


\end{document}